\documentclass[11pt]{amsart}
\usepackage[a4paper,margin=28mm]{geometry}
\usepackage[T1]{fontenc}
\usepackage{lmodern}
\usepackage{needspace}
\usepackage{amsmath,amssymb,amsthm}
\usepackage[hidelinks]{hyperref}
\newtheorem{theorem}{Theorem}[section]
\newtheorem{lemma}[theorem]{Lemma}
\newtheorem{proposition}[theorem]{Proposition}
\newtheorem{corollary}[theorem]{Corollary}
\theoremstyle{definition}
\newtheorem{definition}[theorem]{Definition}
\theoremstyle{remark}
\newtheorem{remark}[theorem]{Remark}
\newcommand{\R}{R}
\newcommand{\T}{\mathcal T}
\newcommand{\supp}{\operatorname{supp}}
\newcommand{\E}{\mathbb E}
\newcommand{\Prob}{\mathbb P}
\newcommand{\B}{\mathcal B}
\newcommand{\D}{\mathcal D}
\DeclareMathOperator{\Impart}{Im}
\DeclareMathOperator{\Repart}{Re}
\title[Divergence of decreasing rearranged Fourier series]{Divergence of decreasing rearranged Fourier series\\
on every infinite compact abelian group}
\author{Morten Nielsen}
\address{Department of Mathematical Sciences\\
  Aalborg University\\
  Thomas Manns Vej 23\\
  DK--9220 Aalborg\\
  Denmark}
\email{mnielsen@math.aau.dk}
\begin{document}

\begin{abstract}
On every infinite compact Hausdorff abelian group there is a
complex-valued continuous function whose Fourier sums, taken over
coefficients above a decreasing magnitude threshold, are unbounded
almost everywhere. In fact, such functions form a dense \(G_\delta\)
in the space of continuous functions with its uniform norm.
Moreover, a divergent example may be chosen with Fourier support
contained in any prescribed infinite subgroup of the dual, while
avoiding any prescribed finite set of frequencies.
No metrizability or zero-dimensionality is assumed.
The proof combines K\"orner's finite circle and Walsh lemmas with a
finite construction on products of odd cyclic groups. A trichotomy
for the dual group, exact quotient transfer, and cyclic sampling
give finite examples on every group under consideration.
\end{abstract}

\maketitle

\section*{Introduction}
A natural way to reconstruct a function from its Fourier coefficients
is to sum the terms in decreasing order of coefficient magnitude.
However, building on Olevskii's methods for constructing divergent rearrangements
of orthogonal series \cite[pp.~59--68]{O75}, K\"orner showed that this
procedure can fail quite strongly: there exists an \(L^2\) function on
the circle whose decreasing rearranged Fourier series diverges unboundedly
almost everywhere \cite{K96}.
He subsequently obtained a continuous example \cite{K99}, and later
established the corresponding divergence phenomenon for the Walsh
system \cite{K06}.

The trigonometric and Walsh systems are both complete systems of
characters of infinite compact abelian groups. This common structure
naturally raises the question whether K\"orner's divergence result
reflects a general property of such character systems. In the present
paper we give an affirmative answer: on every infinite compact
Hausdorff abelian group \(G\), there is a continuous function whose
Fourier sums over coefficients above a decreasing magnitude threshold
are unbounded almost everywhere. In fact, such functions form a dense
\(G_\delta\) in \(C(G)\). No metrizability assumption is needed, and
the conclusion holds for every decreasing ordering of the nonzero
coefficients, regardless of the ordering within ties.

The main new ingredient is a finite construction on products
of odd cyclic groups. The circle and Walsh results alone do not
cover this case: an infinite product of copies of
\(\mathbb Z/p\mathbb Z\), with \(p\) odd, has neither a circle
quotient nor a binary product quotient. We construct uniformly
bounded polynomials on finite products whose threshold maxima are
arbitrarily large on sets of measure arbitrarily close to one.
Together with K\"orner's finite circle and Walsh lemmas, this construction
and a reduction using the dual group cover every infinite compact
abelian group.

Several related results provide context for this question.
C\'ordoba and Fern\'andez \cite[Theorem~1]{CF98} obtained
norm-divergence results on the circle for \(1\leq r<2\).
They also constructed \(L^r\) examples, \(1\leq r<4/3\),
with almost-everywhere unbounded threshold sums
\cite[Section~5, pp.~1137--1138]{CF98}.
In a different direction, Hewitt and Ross \cite{HR74} studied
rearrangements of Fourier-coefficient magnitudes on arbitrary compact
abelian groups. For Vilenkin systems,
Gosselin \cite{G73} proved almost-everywhere convergence in the natural
order under boundedness assumptions, while Gosselin and Young
\cite{GY75} studied prescribed rearrangements preserving convergence.
Episkoposian \cite{E11} studied norm divergence for greedy approximation
with respect to generalized Walsh systems. In contrast, adapting a
construction of Kostyukovsky and Olevskii \cite{KO97}, Nielsen \cite{N07}
constructed a uniformly bounded orthonormal almost greedy basis for
\(L^r(0,1)\), \(1<r<\infty\). Another line of work establishes greedy
convergence or prescribes Fourier-coefficient magnitudes after modifying a
function on a small set; see Episkoposian and Grigorian \cite{EG12},
Grigoryan and Sargsyan \cite{GS18}, and Episkoposian, Grigorian, and
Grigorian \cite{EGG24}. Oniani \cite{O22} studies divergence sets for fixed
enumerations of compact-group characters. 

The paper is organized as follows. Section~\ref{sec:odd}
develops the new finite construction on products of odd cyclic groups,
inspired by the tree methods of K\"orner and Olevskii.
Section~\ref{sec:finite-models} combines this construction with
K\"orner's circle and Walsh inputs and reduces the general
compact-group problem to these three model cases.
Sections~\ref{sec:construction}--\ref{sec:category} develop the
continuous examples and prove the category conclusion. Standard
compact abelian harmonic analysis is used throughout; see \cite{Rudin}.

\section{Statement and Fourier conventions}

Let \(G\) be a compact Hausdorff abelian group with normalized Haar
measure \(\mu_G\): we use the normalized regular Borel probability
measure, completed when discussing almost-everywhere assertions.
Its discrete dual \(\Gamma=\widehat G\) is written
multiplicatively. All function spaces considered are complex-valued. For \(f\in L^2(G)\)
and \(\delta>0\), define
\[
 \widehat f(\gamma)=\int_G f(x)\overline{\gamma(x)}\,d\mu_G(x),
 \qquad
 T_\delta^G f(x)=
 \sum_{\substack{\gamma\in\Gamma\\|\widehat f(\gamma)|\geq\delta}}
       \widehat f(\gamma)\gamma(x).
\]
These sums are finite, as checked below. Superscripts are omitted
when there is no ambiguity. Let $C(G)$ denote the continuous complex-valued functions on $G$ and put
\[
 \D(G)=\left\{f\in C(G):
 \limsup_{\delta\downarrow0}|T_\delta f(x)|=\infty
 \text{ for }\mu_G\text{-almost every }x\right\}.
\]
Our main result is the following. 
\begin{theorem}\label{main:theorem}
If \(G\) is infinite, then \(\D(G)\) is a dense \(G_\delta\) in
\(C(G)\) with its uniform norm. For every \(\rho>0\) and every finite
\(F\subseteq\widehat G\), there is an \(f\in\D(G)\) satisfying
\[
 \|f\|_\infty<\rho,\qquad
 \widehat f(\gamma)=0\quad(\gamma\in F).
\]
If \(G\) is finite, then \(\D(G)\) is empty.
\end{theorem}

In particular, the example may have mean zero. We write
\[
 \supp\widehat f=\{\gamma\in\widehat G:\widehat f(\gamma)\ne0\}
\]
for the nonzero Fourier support; for a character polynomial we also
call this finite set its spectrum.

\Needspace{6\baselineskip}
We first check that the threshold sums and their maximal functions are well
defined, even when \(G\) is not metrizable.

\begin{lemma}\label{main:thresholds}
For \(f\in L^2(G)\), every positive-threshold sum is finite,
\(\supp\widehat f\) is countable, and
\[
 \#\{\gamma:|\widehat f(\gamma)|\geq\delta\}
 \leq\delta^{-2}\|f\|_2^2.
\]
There are at most countably many distinct threshold sums.
For each \(a>0\), the function
\(\sup_{0<\delta<a}|T_\delta f|\) is measurable. Moreover,
\(T_\delta f\to f\) in \(L^2(G)\) as \(\delta\downarrow0\).
\end{lemma}

\begin{proof}
Distinct characters are orthonormal. Bessel's inequality applied
to arbitrary finite subsets gives the cardinality estimate.
The nonzero support is the union of the finite sets where
\(|\widehat f|\geq1/n\), \(n\geq1\), so it is countable.
Finite subsets of a countable set form a countable family;
each threshold sum is indexed by one such subset and is continuous.
This proves the measurability assertion. Completeness of the
characters and Parseval's identity give
\[
 \|f-T_\delta f\|_2^2
 =\sum_{\substack{\gamma\in\supp\widehat f\\
                         |\widehat f(\gamma)|<\delta}}
                  |\widehat f(\gamma)|^2\longrightarrow0,\qquad \text{as }\delta\downarrow0.
\]
Indeed, first make the square-sum outside a finite set small,
and then take \(\delta\) below its least nonzero magnitude.
\end{proof}

\Needspace{6\baselineskip}
The threshold formulation also removes any dependence on how coefficients of
equal magnitude are ordered.

\begin{corollary}\label{main:ties}
Let \(f\in\D(G)\). There is a full-measure set \(X_f\subseteq G\)
such that, for every enumeration \((\gamma_n)_{n\geq1}\) of
\(\supp\widehat f\) satisfying
\[
 |\widehat f(\gamma_{n+1})|\leq|\widehat f(\gamma_n)|,
\]
the partial sums \(\sum_{n=1}^M\widehat f(\gamma_n)\gamma_n(x)\)
are unbounded for every \(x\in X_f\). Unboundedness already holds
along the indices at the ends of complete coefficient-magnitude levels.
\end{corollary}

\begin{proof}
By Lemma~\ref{main:thresholds}, only finitely many coefficients
exceed any positive threshold. Thus the nonzero support can be
enumerated in nonincreasing magnitude, every coefficient is reached,
and each positive magnitude level has finite multiplicity.
Indeed, every nonempty set of remaining magnitudes has a largest
element: if its supremum is \(s>0\), only finitely many coefficients
have magnitude at least \(s/2\), so the supremum is attained among
that finite set.
For \(f\in\D(G)\), the support of \(\widehat f\) is infinite, since a polynomial
has bounded threshold sums. Every nonempty threshold sum is the
partial sum ending at its last complete magnitude level, independently
of the ordering within ties. Take \(X_f\) to be the full-measure set
where the threshold sums are unbounded. No intersection over orderings
is needed, nor is there any need to enumerate an uncountable dual or
insert zero coefficients.
\end{proof}

The theorem therefore gives pointwise divergence despite
\(L^2\) convergence, even for continuous functions.

\section{The finite construction on odd cyclic products}\label{sec:odd}

We now prove the finite odd-cyclic result that supplies the new
model case in the compact-group argument. The construction is stated
for every odd integer, although only odd primes are needed later.
All counting measures are normalized.

\subsection{Statement and normalization}

Fix an odd integer \(p\geq3\), not necessarily prime, and put
\(\R_p=\mathbb Z/p\mathbb Z\) and \(\omega=\exp(2\pi i/p)\).
On \(\R_p^D\), use normalized
counting measure \(\mu_D\), and write
\[
 \chi_\xi(x)=\omega^{\xi\cdot x},\qquad
 \widehat f(\xi)=p^{-D}\sum_{x\in\R_p^D}f(x)\omega^{-\xi\cdot x}.
\]
Thus \(f=\sum_{\xi\in\R_p^D}\widehat f(\xi)\chi_\xi\).
For \(\delta>0\), define
\[
 T_\delta f(x)=
 \sum_{|\widehat f(\xi)|\geq\delta}\widehat f(\xi)\chi_\xi(x),
 \qquad
 T^*f(x)=\max_{\delta>0}|T_\delta f(x)|.
\]
There are only finitely many distinct threshold sums. The following theorem
provides uniformly bounded polynomials whose threshold maximal functions are
large on sets of measure arbitrarily close to one.

\begin{theorem}\label{thm:finite}
Let \(p\geq3\) be odd, \(A>0\), and \(0<\varepsilon<1\). There exist a positive integer
\(D\) and a complex-valued character polynomial \(P\) on \(\R_p^D\)
such that
\[
 \|P\|_\infty\leq1,\qquad
 \mu_D\{x:T^*P(x)>A\}>1-\varepsilon,\qquad
 \widehat P(0)=0.
\]
\end{theorem}
The proof of Theorem~\ref{thm:finite} relies on the technical
lemmas~\ref{lem:chirp}--\ref{lem:moments} below.
\subsection{An exact localized-chirp identity}

\Needspace{6\baselineskip}
The construction requires polynomials that are localized in the group while
all their nonzero Fourier coefficients have the same magnitude.

\begin{lemma}\label{lem:chirp}
Let \(E\subseteq\R_p^n\) and \(d\geq n\). Write
\(X=(x,z)\in\R_p^n\times\R_p^{d-n}\), and let \(Y\in\R_p^d\).
The function
\[
 F_E(X,Y)=1_E(x)\omega^{X\cdot Y}
\]
satisfies \(|F_E|=1_E(x)\) and
\begin{equation}
 \widehat F_E(a,b)=
 p^{-d}1_E(b_1,\ldots,b_n)\omega^{-a\cdot b}.
\end{equation}
In particular, \(|\widehat F_E(a,b)|\in\{0,p^{-d}\}\), also when
\(E=\varnothing\).
\end{lemma}

\begin{proof}
The assertion about \(|F_E|\) is immediate. At a frequency
\((a,b)\in\R_p^d\times\R_p^d\), normalization in both groups gives
\[
 \widehat F_E(a,b)
 =p^{-2d}\sum_X1_E(x)\omega^{-a\cdot X}
       \sum_Y\omega^{(X-b)\cdot Y}.
\]
The inner sum factors into \(d\) one-dimensional character sums:
\[
 \sum_{Y\in\R_p^d}\omega^{(X-b)\cdot Y}
 =\prod_{\ell=1}^d\sum_{y\in\R_p}
       \omega^{(X_\ell-b_\ell)y}.
\]
Each factor is \(p\) if \(X_\ell=b_\ell\) and zero otherwise. Hence
the product is \(p^d\) precisely when \(X=b\). Only that term remains
in the outer sum. Its first \(n\) coordinates equal
\((b_1,\ldots,b_n)\), and therefore
\[
 \widehat F_E(a,b)
 =p^{-d}1_E(b_1,\ldots,b_n)\omega^{-a\cdot b},
\]
as claimed. The final assertion follows because the last factor is
unimodular whenever the indicator is nonzero.
\end{proof}

We shall also use the following two immediate facts. If a function
\(f(x)\) is independent of a fresh coordinate \(s\), then multiplying
it by the character \(\omega^s\) produces the function
\(f(x)\omega^s\), whose nonzero Fourier coefficients all have
frequency one in the \(s\)-coordinate. Moreover, the product of
functions depending on disjoint sets of coordinates has Fourier
coefficients given by the tensor product of their coefficients.
In particular,
\[
 Q(u,v)=\omega^{u\cdot v},\qquad u,v\in\R_p^m,
\]
is unimodular, takes values among the \(p\)-th roots of unity, and has
all its Fourier coefficients of modulus \(p^{-m}\).
When \(m=0\), this means \(Q=1\).

\subsection{Realizing a tree with ordered coefficient levels}
Fix \(N\geq 1\), and let
\[
\T_N
=
\bigcup_{j=0}^{N-1}\{0,\ldots,p-1\}^j,
\qquad
L=|\T_N|
=
\frac{p^N-1}{p-1}.
\]
We use the usual rooted-tree terminology. The root is the empty word $\varnothing$, the depth of a word $v$ is its length $|v|$, and the children of $v$ are the words $va$, $0 \leq a \leq p-1$,  where \(va\) denotes concatenation. The node $v$ is the parent of the nodes $va$. We say that a node $u$ is an ancestor of $v$ if $u$ is an initial segment of $v$, with $u=v$ allowed.

The following lemma realizes the tree by functions with pairwise
disjoint spectra. The rank of a vertex determines the common magnitude
of the nonzero Fourier coefficients of the corresponding function.
Consequently, thresholding the Fourier coefficients will select the
vertices in the prescribed rank order.

We set
\[
 \zeta_a=\omega^a,\qquad a\in\{0,\ldots,p-1\}.
\]

\begin{lemma}\label{lem:realize}
Fix \(N\geq1\) and any bijection \(r:\T_N\to\{1,\ldots,L\}\).
There exist an integer \(D\geq1\), an integer \(C\),
sets \(E_v\subseteq\R_p^D\) for \(|v|\leq N\), a function
\(k:\R_p^D\to\{1,\omega,\ldots,\omega^{p-1}\}\), and functions
\(P_v:\R_p^D\to\mathbb C\) for \(v\in\T_N\), each a character polynomial, with the following properties:
\begin{enumerate}
 \item \(E_\varnothing=\R_p^D\). For every \(v\in\T_N\), the \(p\) sets
       \(E_{va}\), \(0\leq a\leq p-1\), partition \(E_v\). Moreover,
       \(\mu_D(E_v)=p^{-|v|}\) for every \(v\) with \(|v|\leq N\).
 \item For \(v\in\T_N\), the function
       \[
        h_v:\R_p^D\to\{0,\zeta_0,\ldots,\zeta_{p-1}\},\qquad
        h_v=\sum_{a=0}^{p-1}\zeta_a1_{E_{va}},
       \]
       vanishes outside \(E_v\) and equals \(\zeta_a\) on \(E_{va}\);
       and \(P_v=k\,h_v\).
 \item The Fourier supports of the \(P_v\) are pairwise disjoint.
 \item The exponents \(C+r(v)\) are nonnegative, and, for
       \(\xi\in\R_p^D\),
       \[
        |\widehat P_v(\xi)|\in\{0,p^{-(C+r(v))}\}.
       \]
       Moreover, \(\widehat P_v(0)=0\).
       Thus the nonzero coefficients have equal modulus within each
       node, while these moduli strictly decrease as \(r(v)\) increases.
\end{enumerate}
\end{lemma}

\begin{proof}
We proceed in rounds \(j=0,1,\ldots,N-1\). At the end of round \(j\),
the polynomials \(P_v\) have been constructed for all \(|v|\leq j\),
and the sets \(E_v\) have been constructed for all \(|v|\leq j+1\).
Thus the sets needed to build the depth-\((j+1)\) polynomials are
available when the next round begins. Previously constructed objects
are always lifted to an enlarged product group by ignoring the new
coordinates. Their current versions may be multiplied by common
chirps in later rounds; the sets \(E_v\) and the functions \(h_v\)
remain unchanged.

\emph{Round \(0\).}
On \(\R_p\), let \(k=1\), \(P_\varnothing(x)=\omega^x\), and
\(C=-r(\varnothing)\). 
Let \(E_\varnothing=\R_p\), and define its
\(p\) children by
\[
 E_a=\{x\in\R_p:\omega^x=\zeta_a\}=\{a\},
 \qquad a=0,\ldots,p-1,
\]
since \(\zeta_a=\omega^a\) and \(\omega\) is a primitive \(p\)-th
root of unity. These singletons partition \(E_\varnothing\), each
with normalized measure \(1/p\), and
\(P_\varnothing=\sum_a\zeta_a1_{E_a}=\omega^x\).
The only nonzero Fourier coefficient of \(P_\varnothing\) has modulus one.
All assertions hold for the root and its children; in particular,
the root exponent is \(C+r(\varnothing)=0\).
This completes round \(0\): the depth-zero polynomial and the sets
through depth one have been constructed.

\emph{Round \(j\), where \(1\leq j<N\).}
Suppose rounds \(0,\ldots,j-1\) have been completed on \(\R_p^n\).
Thus the sets \(E_v\) are available for \(|v|\leq j\), the
polynomials \(P_v\) are available for \(|v|<j\), and properties
\emph{(1)}--\emph{(4)} hold for all objects constructed so far, with
a common function \(k\) and integer \(C\). Round \(j\) constructs
simultaneously the polynomials \(P_v\), \(|v|=j\), and the sets
\(E_{va}\) associated with the child nodes \(va\), \(|va|=j+1\).

The round has three operations. First, a common chirp shifts all
previous coefficient levels by the same amount. Second, a localized
chirp realizes each depth-\(j\) node on its prescribed set \(E_v\).
Finally, a private coordinate separates its spectrum from every other
node and divides \(E_v\) into \(p\) children of equal measure.

All existing polynomials have coefficient levels
\(p^{-(C+r(v))}\). Choose an integer \(M\geq0\) large enough that
\begin{equation}
 d_v:=C+M+r(v)\geq n\qquad (|v|=j).
\end{equation}
For example, take
\[
 M=\max\bigl\{0,\ n-C-\min_{|v|=j}r(v)\bigr\}.
\]
On \(2M\) fresh coordinates introduce the unimodular chirp
\(Q(u,w)=\omega^{u\cdot w}\). Replace
\[
 k\ \text{by}\ k'=Qk,\qquad
 P_v\ \text{by}\ P'_v=QP_v,\quad |v|<j,
 \qquad C\ \text{by}\ C'=C+M.
\]
The old identities \(P'_v=k'h_v\) hold with the \emph{same}
functions \(h_v\) and the same sets. By Lemma~\ref{lem:chirp} with
\(E=\R_p^M\), every coefficient of \(Q\) has modulus \(p^{-M}\).
Since the old variables and the variables of \(Q\) are disjoint,
the Fourier coefficients of \(QP_v\) are tensor products of those
of \(Q\) and \(P_v\). Thus every old coefficient level is multiplied
by \(p^{-M}\). Tensoring each old spectrum with the same set of new
frequencies also preserves pairwise disjointness.
Since \(M\geq0\), no old exponent decreases; each new exponent
\(C'+r(v)=d_v\) is at least \(n\geq1\).
This verifies nonnegativity of all exponents at every round.

For \(|v|<j\), \(P'_v\) denotes the leveled-up polynomial \(QP_v\)
just defined. We now reuse the same primed notation for each new
depth-\(j\) node, but with a different defining formula, since no
old polynomial exists yet at such a node; both cases are unprimed
together at the end of the round.

For each new node \(v\) of depth \(j\), introduce its own fresh
coordinates
\[
 z_v\in\R_p^{d_v-n},\qquad
 Y_v\in\R_p^{d_v},\qquad s_v\in\R_p.
\]
All these coordinate blocks are disjoint, both from one another
and from those of \(Q\). Let \(x\in\R_p^n\) denote the entire
pre-round ambient variable, and set \(X_v=(x,z_v)\). Thus \(x\)
contains only the coordinates present before the common-chirp extension,
so \(X_v\) does not include the newly introduced coordinates \((u,w)\).
 We regard \(E_v\)
as a cylinder set in the enlarged product group and define
\begin{equation}
 P'_v=1_{E_v}(x)\omega^{X_v\cdot Y_v+s_v}.
\end{equation}
Consequently, \(P'_v\) is independent of \((u,w)\).
Ignoring the last character for a moment, Lemma~\ref{lem:chirp}
gives coefficient magnitudes \(0\) or \(p^{-d_v}\). Multiplication by
\(\omega^{s_v}\) merely shifts the frequency in the private coordinate
\(s_v\) from zero to one, without changing these magnitudes. Hence
\(P'_v\) has exactly the desired coefficient level
\(p^{-d_v}=p^{-(C'+r(v))}\). Its independence of all other newly
introduced coordinate blocks forces frequency zero in those blocks.

The coordinate \(s_v\) verifies spectral disjointness without any
assumption on the shape of \(E_v\). Every nonzero coefficient of
\(P'_v\) has \(s_v\)-frequency one. All old polynomials and all
other polynomials introduced in this round are independent of
\(s_v\), and therefore
have frequency zero there. Consequently the new spectra are
pairwise disjoint and are disjoint from every old spectrum.
In later rounds, tensoring existing polynomials with a common chirp
on fresh coordinates preserves every separation already established.
A subsequently introduced node \(t\) is separated from all existing
nodes by its own fresh coordinate \(s_t\); it need not have
frequency zero in the earlier coordinate \(s_v\).
The same observation proves \(\widehat P'_v(0)=0\).
The \(s_v\)-frequency of an existing nonroot polynomial remains one
under every later multiplication by a fresh-coordinate chirp.
For the root, the original \(x\)-frequency similarly remains one.
Thus the zero-mean property also persists throughout the construction.

It remains to define the new children. 
We regard the previously constructed set \(E_v\) and function \(k'\) on the fully enlarged product group by cylindrical extension. On \(E_v\), put
\[
 h'_v:=\overline{\bigl(k'\bigr)}\,P'_v,
\]
where the bar denotes complex conjugation.
Define
\[
 E_{va}=\{x_{\rm full}\in E_v:h'_v(x_{\rm full})=\zeta_a\}.
\]
Since \((k')^p=1\),
\(\overline{k'}=(k')^{-1}=(k')^{p-1}\), we obtain that \(P'_v=k'h'_v\).
Outside \(E_v\), put \(h'_v=0\).
On \(E_v\), the function \(h'_v\) is a \(p\)-th root of unity.
Conditioning on every coordinate except \(s_v\), it is a fixed
\(p\)-th root of unity times \(\omega^{s_v}\), since \(k'\) is
independent of \(s_v\). As \(s_v\) runs uniformly through \(\R_p\),
each \(p\)-th root occurs exactly once. Thus each child has exactly
\(1/p\) of its parent's measure. The children partition \(E_v\), and,
because \(h'_v=\zeta_a\) on \(E_{va}\),
\[
 h'_v=\sum_{a=0}^{p-1}\zeta_a1_{E_{va}}.
\]

This step uses
\[
 n'=n+2M+\sum_{|v|=j}(2d_v-n+1)
\]
coordinates in total: the common chirp uses \(2M\), while the node
\(v\) uses \((d_v-n)+d_v+1=2d_v-n+1\) private coordinates. This
number is finite.
Previously defined children are unchanged. Drop the primes from the
updated objects and regard them as living on \(\R_p^{n'}\). Properties
\emph{(1)}--\emph{(4)} now hold with the polynomials constructed
through depth \(j\) and the sets constructed through depth \(j+1\).
This completes round \(j\).

After round \(N-1\), the polynomials have been constructed for every
\(v\in\T_N\), and the sets \(E_v\) are available through depth \(N\).
Set \(D\) to be the value of \(n\) reached after round \(N-1\), so
that all constructed objects live on \(\R_p^D\), as claimed.
Under the successive cylinder extensions, the initial root set
\(E_\varnothing=\R_p\) has become
\(\R_p\times\R_p^{D-1}=\R_p^D\), in agreement with property~\emph{(1)}.
\end{proof}

\begin{remark}\label{rem:important}
The common multiplier is essential: it preserves the pointwise
tree cancellation, whereas disjoint spectra ensure that coefficients
of different nodes never add. The integer \(C\) shifts in each
round, but the differences \(r(v)-r(w)\) never change. Hence new
coefficient levels can be inserted between old ones.
Equal modulus is required within each node, not across the entire
tree. No perturbation of equal coefficients is needed.

For this particular implementation, let \(n_0=1\) and let \(n_j\)
be the ambient dimension after adding the nodes of depth \(j\).
Since \(d_v\geq n_{j-1}\) and \(M\geq0\), the coordinate count gives
\[
 n_j\geq(1+p^j)n_{j-1},\qquad
 D=n_{N-1}\geq\prod_{j=1}^{N-1}(1+p^j)
       \geq p^{N(N-1)/2}.
\]
This describes the growth of the construction above. The construction
could very likely be optimized further, but we do not pursue this question here.
\end{remark}

\subsection{The tree order and a stopped random walk}
\label{sec:tree-order}

The construction of Lemma~\ref{lem:realize} accepts an \emph{arbitrary}
bijection \(r:\T_N\to\{1,\ldots,L\}\). We now fix a particular
bijection \(r_0\) such that reading the coefficients of the \(P_v\) in
decreasing order of magnitude also reads their imaginary parts in a
controlled order. This ordering property is used in Lemma~\ref{lem:order}
below.

Put \(b_p=(p-1)/2\), so that \(p=2b_p+1\) labels split into three
groups by the sign of \(\Impart\zeta_a\): \textit{positive }for
\(1\leq a\leq b_p\), \textit{zero} for \(a=0\), and \textit{negative }for
\(b_p+1\leq a\leq p-1\).

For \(v\in\T_N\), let \(\mathrm{Order}(v)\) denote the list of every
node in \(v\)'s subtree (within \(\T_N\)), defined recursively. At the
last generation, \(\mathrm{Order}(v)=[\,v\,]\) (that is, \(|v|=N-1\)).
Otherwise \(|v|<N-1\) and \(\mathrm{Order}(v)\) is the concatenation,
in the order listed, of
\[
 \begin{aligned}[t]
  &\mathrm{Order}(v1),\ \ldots,\ \mathrm{Order}(vb_p),\quad
   \mathrm{Order}(v0),\quad [\,v\,],\\
  &\mathrm{Order}\bigl(v(b_p{+}1)\bigr),\ \ldots,\ \mathrm{Order}(v(p{-}1)).
 \end{aligned}
\]
Applying this at the root gives a single list of all of \(\T_N\); let
\(r_0(v)\) be the position of \(v\) in that list.

Equivalently: order the \(p\) children before or after their parent
according to the sign of their label's imaginary part (positive and
zero before, negative after), and order siblings among themselves the
same way, recursively.

The purpose of this particular order $r_0$ is to turn a large sum of absolute imaginary
increments into a large partial sum.

\begin{lemma}\label{lem:order}
Fix a terminal word of length \(N\). Among its ancestor nodes,
ordered by \(r_0\), all strictly negative imaginary increments
precede all strictly positive imaginary increments.
Consequently, for any nonnegative numbers \(a_v\) and any point \(x\),
\begin{equation}\label{eq:variation}
 \max_{1\leq t\leq L}
 \left|\sum_{r_0(v)\leq t}a_vh_v(x)\right|
 \geq\frac13\sum_{v\in\T_N}a_v|\Impart h_v(x)|.
\end{equation}
\end{lemma}

\begin{proof}
For a fixed \(x\), the nodes \(v\) for which \(h_v(x)\ne0\) are
precisely the \(N\) ancestors on the unique terminal path determined
by \(x\), one at each depth \(0,\ldots,N-1\). The key point: if \(v\)
is one such ancestor, every ancestor strictly deeper than \(v\) lies
inside the child subtree that the path enters at \(v\), i.e., it lies
further along that same path. By the ordering rule of
Subsection~\ref{sec:tree-order}, this
subtree is listed \emph{before} \(v\) when the increment at \(v\) has
label \(1,\ldots,b_p\) (positive imaginary part), and \emph{after}
\(v\) when the label is \(b_p{+}1,\ldots,p{-}1\) (negative imaginary
part); the label \(0\) has zero imaginary part and never enters the
comparison below.

Now let \(v,u\) be any two ancestors on the path with
\(\Impart h_v(x)<0<\Impart h_u(x)\). If \(v\) is the shallower of the
two, then \(u\) lies in the subtree entered at \(v\). Since the
increment at \(v\) is negative, this subtree is listed \emph{after}
\(v\), so
\(r_0(u)>r_0(v)\). If instead \(u\) is the shallower node, \(v\) lies
in the subtree entered at \(u\). Since the increment at \(u\) is
positive, this subtree is listed \emph{before} \(u\), so
\(r_0(v)<r_0(u)\). Thus \(v\) precedes \(u\) in either case, proving
the first part of the lemma.

For a real sequence with all negative terms before all positive
terms, let the negative total be \(-U\) and the positive total
be \(V\), where \(U,V\geq0\). Its maximum absolute partial sum
is at least \(\max\{U,|V-U|\}\),
\[
 U+V=2U+(V-U)\leq 2U+|V-U|
       \leq3\max\{U,|V-U|\}.
\]
We apply this to the imaginary parts of the ordered summands, use the
elementary estimate \(|\Impart z|\leq |z|\), and 
\eqref{eq:variation} follows.
\end{proof}

By Lemma~\ref{lem:realize}, the terminal sets $E_v$ all have measure
\(p^{-N}\). Consequently the sequence \(Z_1,\ldots,Z_N\) of
increments along the path of a uniformly chosen point consists of
independent random variables, each uniform on
\(\{1,\omega,\ldots,\omega^{p-1}\}\). In particular,
\[
 \E Z_j=0,\qquad |Z_j|=1,\qquad
 S_\ell:=\sum_{j=1}^{\ell}Z_j,\qquad \E|S_N|^2=N.
\]

Fix \(\beta>1\). If \(v=(v_1,\ldots,v_j)\) has depth \(j\), define
\[
 S_\ell(v)=\sum_{r=1}^{\ell}\zeta_{v_r}
 \quad(0\leq\ell\leq j),\qquad S_0(v)=0,
\]
and put
\[
 a_v=
 \begin{cases}
  1,&\text{if }\max_{0\leq\ell\leq j}|S_\ell(v)|\leq\beta-1,\\
  0,&\text{otherwise}.
 \end{cases}
\]
Thus \(a_v\in\{0,1\}\) is a fixed scalar attached to the word \(v\),
not a function of \(x\). The sums \(S_\ell(v)\) are exactly the
preceding path sums at every point of \(E_v\).

The following stopping observation explains why these coefficients keep
the full polynomial uniformly bounded.

\begin{lemma}\label{lem:stopping}
With the coefficients \(a_v\) defined above, let
\[
 \tau=\min\{\ell\in\{1,\ldots,N\}:|S_\ell|>\beta-1\},
 \qquad \min\varnothing=\infty,
\]
at each point of the underlying product group. Then
\[
 \sum_{v\in\T_N}a_vh_v=S_{\tau\wedge N}
\]
and, at every point,
\begin{equation}
 \left|\sum_{v\in\T_N}a_vh_v\right|\leq \beta.
\end{equation}
\end{lemma}

\begin{proof}
Fix a point \(x\), and let \(v_0=\varnothing,v_1,\ldots,v_{N-1}\) be
its ancestors, so that \(h_{v_j}=Z_{j+1}\) at \(x\). Because \(v_j\)'s
own labels are exactly \(x\)'s first \(j\) increments,
\(S_\ell(v_j)=S_\ell\) at \(x\) for every \(\ell\leq j\), and hence
\[
 a_{v_j}=1
 \iff \max_{0\leq \ell\leq j}|S_\ell|\leq\beta-1
 \iff j<\tau.
\]
Consequently
\[
 \sum_{v\in\T_N}a_vh_v
 =\sum_{j=0}^{N-1}a_{v_j}Z_{j+1}
 =\sum_{j=0}^{N-1}1_{\{j<\tau\}}Z_{j+1}
 =\sum_{k=1}^{\tau\wedge N}Z_k
 =S_{\tau\wedge N},
\]
which is the claimed identity, evaluated at \(x\).

For the bound: if \(\tau>N\), no partial sum through step \(N\) ever
exceeded \(\beta-1\), so \(|S_{\tau\wedge N}|=|S_N|\leq\beta-1\). If
\(\tau\leq N\), minimality of \(\tau\) gives \(|S_{\tau-1}|\leq\beta-1\)
(with \(S_0=0\)), while \(|Z_\tau|=1\); the triangle inequality gives
\[
 |S_{\tau\wedge N}|=|S_\tau|
 \leq|S_{\tau-1}|+|Z_\tau|
 \leq(\beta-1)+1=\beta.
\]
Either way \(|S_{\tau\wedge N}|\leq\beta\), at every point \(x\).
\end{proof}

We also need to know that the stopping rule discards only a small set
of paths. The next estimate is the required maximal inequality.

\begin{lemma}\label{lem:crossing}
For the random walk above,
\begin{equation}
 \Prob\{\max_{0\leq\ell\leq N}|S_\ell|>\beta-1\}
 \leq\frac{N}{(\beta-1)^2}.
\end{equation}
\end{lemma}

\begin{proof}
With respect to the filtration generated by \(Z_1,\ldots,Z_\ell\),
\(|S_\ell|^2\) is a nonnegative submartingale. Doob's maximal
inequality gives
the asserted estimate since \(\E|S_N|^2=N\). For completeness, it
also follows directly from the stopping time in
Lemma~\ref{lem:stopping}. Independence
and zero means of future increments give
\(\E(|S_N|^2\mid Z_1,\ldots,Z_\ell)=|S_\ell|^2+N-\ell\).
Multiply by \(1_{\{\tau=\ell\}}\), take expectations, and sum over
\(1\leq\ell\leq N\). Nonnegativity then yields
\((\beta-1)^2\Prob\{\tau\leq N\}\leq\E|S_N|^2\), as required.
\end{proof}

\Needspace{6\baselineskip}
To apply Lemma~\ref{lem:order} on a set of large measure, we need a lower
bound for the accumulated absolute imaginary increments.

\begin{lemma}\label{lem:moments}
If \(Z\) is uniform on the \(p\)-th roots of unity and
\(W=|\Impart Z|\), then
\[
 \E W^2=\frac12,\qquad \E W\geq\frac12,\qquad
 \operatorname{Var}(W)\leq\frac14.
\]
Consequently, for independent copies \(W_1,\ldots,W_N\),
\begin{equation}\label{eq:concentration}
 \Prob\left\{\sum_{j=1}^N W_j<N/4\right\}\leq\frac4N.
\end{equation}
All these bounds are independent of \(p\).
\end{lemma}

\begin{proof}
Since \(p\) is odd, multiplication by \(2\) permutes
\(\mathbb Z/p\mathbb Z\), including when \(p\) is composite. Thus
\[
 \E Z^2=\frac1p\sum_{a=0}^{p-1}\omega^{2a}
        =\frac1p\sum_{a=0}^{p-1}\omega^a=0.
\]
Hence
\[
 \E W^2=\E(\Impart Z)^2
       =\frac{1-\Repart(\E Z^2)}2=\frac12.
\]
As \(0\leq W\leq1\), we have \(\E W\geq\E W^2=1/2\), and
\(\operatorname{Var}(W)=1/2-(\E W)^2\leq1/4\).
The sum of \(N\) independent copies has mean at least \(N/2\)
and variance at most \(N/4\). Chebyshev's inequality at deviation
\(N/4\) proves~\eqref{eq:concentration}.
\end{proof}

On the event with no crossing, every coefficient \(a_v\) belonging
to the path of the point equals one. Nodes off that path do not
contribute, since their \(h_v\) vanish at the point. If also
\(\sum_{j=1}^N|\Impart Z_j|\geq N/4\), then~\eqref{eq:variation} gives
\begin{equation}\label{eq:largesum}
 \max_{1\leq t\leq L}
 \left|\sum_{r_0(v)\leq t}a_vh_v\right|\geq N/12.
\end{equation}
By Lemmas~\ref{lem:crossing} and \ref{lem:moments}, the exceptional
set has measure at most \(N/(\beta-1)^2+4/N\).

\subsection{Exact threshold selection}

\begin{proof}[Proof of Theorem~\ref{thm:finite}]
Choose an integer
\begin{equation}\label{eq:Nchoice}
 N>\max\{8/\varepsilon,\;1152A^2/\varepsilon\}.
\end{equation}
Apply Lemma~\ref{lem:realize} with \(r=r_0\), and set
\[
 \beta=1+\sqrt{2N/\varepsilon},
 \qquad P=\beta^{-1}\sum_{v\in\T_N}a_vP_v.
\]
Since \(P_v=kh_v\) with \(|k|=1\),
Lemma~\ref{lem:stopping} gives \(\|P\|_\infty\leq1\).
Moreover, every \(P_v\) has zero Fourier coefficient at the identity,
so \(\widehat P(0)=0\).

Disjointness of the spectra and \(a_v\in\{0,1\}\) show that
each surviving coefficient from node \(v\) has modulus
\(\beta^{-1}p^{-(C+r_0(v))}\). Thus for \(1\leq t\leq L\),
\begin{equation}\label{eq:threshold}
 T_{\beta^{-1}p^{-(C+t)}}P
   =\beta^{-1}\sum_{r_0(v)\leq t}a_vP_v
   =\beta^{-1}k\sum_{r_0(v)\leq t}a_vh_v.
\end{equation}
This is an exact identity, including any ties within a node.
There are no ties between different active nodes.
If the node of rank \(t\) is inactive, its nominal coefficient level
does not occur in \(P\); the identity remains valid because its
contribution on the right is zero.

The exceptional measure in~\eqref{eq:largesum} is at most
\[
 \frac{N}{(\beta-1)^2}+\frac4N
   =\frac{\varepsilon}{2}+\frac4N<\varepsilon.
\]
At every other point,~\eqref{eq:threshold} yields
\[
 T^*P\geq\frac{N}{12(1+\sqrt{2N/\varepsilon})}
       \geq\frac{N}{24\sqrt{2N/\varepsilon}}
       =\frac1{24}\sqrt{\frac{\varepsilon N}{2}}
       =\frac{\sqrt{\varepsilon N}}{24\sqrt2}>A.
\]
Here \(\sqrt{2N/\varepsilon}\geq1\), and the last inequality
is precisely the second requirement in~\eqref{eq:Nchoice}, since
\(1152=(24\sqrt2)^2\).
\end{proof}

\section{Finite models and reduction to compact abelian groups}\label{sec:finite-models}

\subsection{The finite property and exact transfer}\label{sec:transfer}

For a character polynomial \(P\), let
\(T^*_GP=\max_{\delta>0}|T_\delta^GP|\).
If \(\lambda_1>\cdots>\lambda_r>0\) are the distinct nonzero
coefficient magnitudes of \(P\), then
\[
 T^*_GP=\max\bigl(\{0\}\cup
           \{|T_{\lambda_j}^GP|:1\leq j\leq r\}\bigr).
\]
Indeed, every nonempty set of frequencies selected by a threshold is
also selected at one of these coefficient levels. Thus the maximum is
attained and \(T^*_GP\) is continuous; when \(P=0\), the list is empty
and \(T^*_GP=0\). Consequently, if \(T^*_GP(x)>A\geq0\), then
\(|T_{\lambda_j}^GP(x)|>A\) at some nonzero coefficient level
\(\lambda_j\).

\begin{definition}
We say that \(\B(G)\) holds if, for every \(A>0\) and
\(0<\varepsilon<1\), there is a character polynomial \(P\) with
\[
 \|P\|_\infty\leq1,\qquad
 \mu_G\{T^*_GP>A\}>1-\varepsilon.
\]
\end{definition}

\Needspace{6\baselineskip}
To use examples constructed on simpler groups, we need to transfer them
through quotient maps without changing their threshold sums.

\begin{lemma}\label{main:quotient}
Let \(q:G\to H\) be a continuous surjective homomorphism of compact
Hausdorff abelian groups. For \(f\in C(H)\),
\[
 \widehat{f\circ q}(\gamma)=
 \begin{cases}
 \widehat f(\lambda),&\gamma=\lambda\circ q,\quad\lambda\in\widehat H,\\
 0,&\gamma\text{ is nontrivial on }\ker q.
 \end{cases}
\]
These alternatives exhaust \(\widehat G\), and the descending
character is unique. Consequently
\[
 T_\delta^G(f\circ q)=(T_\delta^Hf)\circ q.
\]
Property \(\B(H)\) implies \(\B(G)\), and \(f\in\D(H)\) implies
\(f\circ q\in\D(G)\).
\end{lemma}

\begin{proof}
Let \(N=\ker q\). The pushforward \(q_*\mu_G\) is a
translation-invariant probability measure on \(H\), and hence
\[
 q_*\mu_G=\mu_H.
\]

Suppose first that \(\gamma\in\widehat G\) is trivial on \(N\). Then
\[
 \lambda(q(x)):=\gamma(x)
\]
defines a character \(\lambda\in\widehat H\). It is well defined and
continuous because the induced continuous bijection \(G/N\to H\) is a
homeomorphism. The character \(\lambda\) is unique because \(q\) is
surjective. Using the pushforward identity gives
\[
 \widehat{f\circ q}(\gamma)
 =\int_G f(q(x))\overline{\lambda(q(x))}\,d\mu_G(x)
 =\int_H f(z)\overline{\lambda(z)}\,d\mu_H(z)
 =\widehat f(\lambda).
\]

If instead \(\gamma\) is nontrivial on \(N\), choose \(y\in N\) with
\(\gamma(y)\neq1\), and put
\[
I=\widehat{f\circ q}(\gamma).
\]
Translation invariance and \(q(y)=e_H\) give
\[
 I=\overline{\gamma(y)}I,
\]
so \(I=0\). Thus
\[
 \widehat{f\circ q}(\gamma)
 =
 \begin{cases}
  \widehat f(\lambda),
   & \gamma=\lambda\circ q
     \text{ for a unique }\lambda\in\widehat H,\\
  0,
   & \gamma\text{ is nontrivial on }N.
 \end{cases}
\]

There is therefore no change in either the values or the magnitudes of
the nonzero Fourier coefficients under pullback, and for every
\(\delta>0\),
\[
 T_\delta^G(f\circ q)(x)
 =
 \sum_{\substack{\lambda\in\widehat H\\
                  |\widehat f(\lambda)|\ge\delta}}
 \widehat f(\lambda)(\lambda\circ q)(x)
 =T_\delta^Hf(q(x)).
\]
Hence
\[
T_\delta^G(f\circ q)=(T_\delta^Hf)\circ q.
\]

Finally, surjectivity gives
\[
 \|f\circ q\|_\infty=\|f\|_\infty,
\]
while \(q_*\mu_G=\mu_H\) preserves the measures of threshold
superlevel sets and divergence sets. Consequently, \(\B(H)\) implies
\(\B(G)\), and \(f\in\D(H)\) implies \(f\circ q\in\D(G)\).
\end{proof}

\begin{lemma}\label{main:infinitedual}
If \(G\) is infinite, then \(\widehat G\) is infinite.
\end{lemma}

\begin{proof}
If \(\widehat G\) were finite, every character would have finite
order and hence finite image. Character separation makes the map
\[
 G\longrightarrow\prod_{\gamma\in\widehat G}\gamma(G),
 \qquad x\longmapsto(\gamma(x))_{\gamma\in\widehat G}
\]
injective. Its target would be finite, contradicting infinitude of \(G\).
\end{proof}

\Needspace{6\baselineskip}
When assembling the finite examples, we will need disjoint spectra.
Multiplication by a character provides this separation without changing the
size of the threshold sums.

\begin{lemma}\label{main:modulation}
For \(\eta\in\Gamma\) and \(f\in C(G)\),
\[
 \widehat{\eta f}(\gamma)=\widehat f(\eta^{-1}\gamma),
 \qquad T_\delta(\eta f)=\eta T_\delta f.
\]
If \(\Gamma\) is infinite and \(S,F\subset\Gamma\) are finite,
some \(\eta\in\Gamma\) satisfies \(\eta S\cap F=\varnothing\).
Modulation preserves uniform norms and threshold-sum magnitudes.
\end{lemma}

\begin{proof}
The first identity follows from the coefficient integral and
the second by reindexing. The only forbidden multipliers are
in the finite set \(FS^{-1}\). Finally, \(|\eta|=1\).
\end{proof}

We shall also use
\begin{equation}\label{main:scaling}
 T_\delta(cf)=cT_{\delta/c}f\quad(c>0),\qquad
 |\widehat f(\gamma)|\leq\|f\|_\infty.
\end{equation}
General unimodular multiplication does not have the threshold
covariance of character modulation.

\subsection{The three finite model inputs}

Let \(\mathbb T\) denote the unit circle with Haar probability measure
\(\mu_{\mathbb T}\), and put \(K_p=\prod_{j\geq1}(\mathbb Z/p\mathbb Z)\).
The latter always has its compact product topology and Haar measure.
In Propositions~\ref{main:circle} and~\ref{main:walsh}, we adapt
K\"orner's results for the trigonometric system \cite{K99} and the
Walsh system \cite{K06} to the present setting.
Whenever an imported result supplies a non-strict lower bound,
we choose its parameters so that this bound is strictly larger
than the target \(A\) in \(\B(G)\).

\begin{proposition}\label{main:circle}
For every \(s>0\) and \(K>1\), there are a trigonometric polynomial
\(P\) and a measurable set \(E\subseteq\mathbb T\) such that
\[
 \|P\|_\infty\leq s,\qquad \mu_{\mathbb T}(E)\leq s,\qquad
 T^*_{\mathbb T}P>K\quad\text{on }\mathbb T\setminus E.
\]
In particular, \(\B(\mathbb T)\) holds.
\end{proposition}

This is \cite[Lemma 7, p.~9]{K99}, with the source's parameters
\(\epsilon,K\) and exceptional set \(B\) renamed \(s,K,E\).
The identification \(t\mapsto e^{2\pi it}\) preserves probability
measure and the threshold convention \(|\widehat P(r)|\geq\delta\);
no rescaling is needed. Continuity of \(P\) identifies its essential
supremum with its uniform norm. Taking
\(0<s<\min(1,\varepsilon)\) and \(K>\max(1,A)\) gives
\(\B(\mathbb T)\).

\begin{proposition}\label{main:walsh}
For each \(0<e<1\), there exist \(\kappa(e)>0\)
and an integer \(N_0(e)\) such that, for every \(N\geq N_0(e)\),
a finite Walsh polynomial \(Q\) satisfies
\[
 \|Q\|_\infty\leq1,\qquad
 \mu_{K_2}\{T^*_{K_2}Q\geq\kappa(e)\sqrt N\}\geq1-e.
\]
In particular, \(\B(K_2)\) holds.
\end{proposition}

In \cite[Theorem 6.5, p.~723]{K06}, take the source parameter
\(\epsilon=e\) and set \(\kappa(e)=\kappa_0(e)\), retaining
\(N_0(e)\) and \(N\). The theorem gives the stated bounds for every
integer \(N\geq N_0(e)\), with \(\widehat Q(u)=0\) for
\(u>2^{4N}\); its additional coefficient bounds involving
\(\alpha(e)\) are not needed here.

The source uses Walsh functions on \([0,1)\) with Lebesgue measure
and half-open dyadic intervals \cite[p.~709]{K06}. The support bound
makes \(Q\) depend on at most \(4N+1\) binary digits. These digit
patterns have the same uniform distribution on the interval and
on \(K_2\), so Fourier coefficients and threshold superlevel measures
agree. Since each pattern has positive measure, the interval
essential supremum also equals the compact-group uniform norm.
Finally, choose \(0<e<\varepsilon\) and \(N\geq N_0(e)\) with
\(\kappa(e)\sqrt N>A\) to obtain \(\B(K_2)\).

\Needspace{6\baselineskip}
The remaining finite model concerns odd cyclic products. The construction
proved in Section~\ref{sec:odd} gives the required input directly.

\begin{proposition}\label{main:oddinput}
For every odd integer \(p\geq3\), \(A>0\), and \(0<\varepsilon<1\),
there exist a positive integer \(D\) and a character polynomial \(P\) on
\((\mathbb Z/p\mathbb Z)^D\) with
\[
 \|P\|_\infty\leq1,\qquad \mu_D\{T^*P>A\}>1-\varepsilon.
\]
Consequently \(\B(K_p)\) holds.
\end{proposition}

\begin{proof}
The finite assertion is Theorem~\ref{thm:finite}. Lift by the projection of \(K_p\) onto
its first \(D\) coordinates and apply Lemma~\ref{main:quotient}.
\end{proof}

Only prime \(p\) will be needed in the main theorem, and no
constants uniform in \(p\) are required for that reduction.

These three models yield \(\B(G)\) for every infinite compact
abelian group in Proposition~\ref{main:allB}, by cyclic sampling
and quotient transfer. The new construction is needed to cover the
odd-prime case; Sections~\ref{sec:construction}--\ref{sec:category} then give the continuous example
and the category conclusion.

\subsection{Reduction using the dual group}

\Needspace{6\baselineskip}
Let $G$ be a compact Hausdorff abelian group. Characters of large finite order allow us to use the circle construction
even when \(G\) has no circle quotient.

\begin{lemma}\label{main:sampling}
If \(\widehat G\) contains characters of arbitrarily large finite
order, then \(\B(G)\) holds.
\end{lemma}

\begin{proof}
Fix \(A,\varepsilon\) and use Proposition~\ref{main:circle} to choose
\[
 P(z)=\sum_{k=-d_0}^{d_0} a_kz^k,\qquad \|P\|_\infty\leq1,\qquad
 \mu_{\mathbb T}(U)>1-\varepsilon/2,\quad U=\{T^*_{\mathbb T}P>A\}.
\]
The set \(U\) is open. Write \(\nu_M\) for uniform measure on
the \(M\)-th roots of unity. For a fixed nonzero integer \(k\),
\(\int z^k\,d\nu_M=0\) whenever \(M>|k|\), whereas
\(\int1\,d\nu_M=1\). Uniform approximation by trigonometric
polynomials therefore shows that \(\nu_M\to \mu_{\mathbb T}\) weakly.

Here only a lower bound for the open set \(U\) is needed;
there is no assumption on its boundary. By regularity, choose a compact
set \(K\subseteq U\) such that
\(\mu_{\mathbb T}(K)>1-\varepsilon\). By Urysohn's lemma, there is a
function \(\varphi\in C(\mathbb T)\) satisfying
\[
 0\leq\varphi\leq 1_U
 \qquad\text{and}\qquad
 \varphi=1\quad\text{on }K.
\]
Consequently,
\[
 \int_{\mathbb T}\varphi\,d\mu_{\mathbb T}
 \geq\mu_{\mathbb T}(K)>1-\varepsilon.
\]
Weak convergence gives
\(\nu_M(U)\geq\int\varphi\,d\nu_M>1-\varepsilon\) for all sufficiently
large \(M\).

Choose a character \(\chi\) of order \(M\) this large, also with
\(M>2d_0\), and set \(Q=P\circ\chi\).
Its image is exactly the group of \(M\)-th roots, and Haar measure
pushes forward to \(\nu_M\). The characters \(\chi^k\),
\(-d_0\leq k\leq d_0\), are distinct: equality would require
\(M\mid k-\ell\). Orthogonality gives exactly the coefficient
list \(a_k\), without aliasing, and hence
\[
 T_\delta^GQ(x)=T_\delta^{\mathbb T}P(\chi(x)).
\]
Thus \(\|Q\|_\infty\leq1\) and
\(\mu_G\{T^*_GQ>A\}=\nu_M(U)>1-\varepsilon\).
\end{proof}

\Needspace{6\baselineskip}
When character orders are bounded, the following algebraic observation will
lead instead to a quotient of the form \(K_p\).

\begin{lemma}\label{main:bounded}
Every infinite abelian group of bounded exponent contains
\(\bigoplus_{j\geq1}\mathbb Z/p\mathbb Z\) as a subgroup
for some prime \(p\).
\end{lemma}

\begin{proof}
Use additive notation. We first isolate an infinite component whose
exponent is a power of a single prime. If \(L\Gamma=0\) and
\(L=\prod_{i=1}^s p_i^{r_i}\), Chinese remainder projections give
\(\Gamma=\bigoplus_{i=1}^s A_i\), where \(p_i^{r_i}A_i=0\).
Explicitly choose integers \(e_i\) equal to \(1\) modulo
\(p_i^{r_i}\) and \(0\) modulo the other factors, and let
\(A_i=e_i\Gamma\). Their sum is the identity and their pairwise
products vanish on \(\Gamma\), giving this decomposition.
At least one component \(A\), killed by \(p^r\), is infinite.

It remains to find infinitely many elements of order \(p\) in this
component, which is why we consider its subgroup \(A[p]\).
We claim \(A[p]=\{a:pa=0\}\) is infinite. Inductively, a group
killed by \(p^r\) with finite \(A[p]\) must be finite:
this is immediate for \(r=1\); for \(r>1\), the group \(pA\)
has smaller exponent and \((pA)[p]\subseteq A[p]\), so is finite
by induction. The map \(A\to pA\) then has finite kernel
\(A[p]\) and finite image, making \(A\) finite.

Having obtained an infinite elementary abelian \(p\)-subgroup, we can
now extract the desired direct sum.
The infinite vector space \(A[p]\) over \(\mathbb F_p\)
contains a countable independent sequence: at each step choose
outside the finite span of preceding elements. Its span is the
required direct sum.
\end{proof}

\begin{lemma}\label{main:primequotient}
If \(\widehat G\) contains
\(\bigoplus_{j\geq1}\mathbb Z/p\mathbb Z\), then \(G\) has
a continuous quotient map onto \(K_p\).
\end{lemma}

\begin{proof}
Let \(\chi_j\) be independent generators of order \(p\).
Identify the \(p\)-th roots with \(\mathbb Z/p\mathbb Z\);
the map \(q(x)=(\chi_j(x))_{j\geq1}\) is a continuous homomorphism.
Its image is compact, hence closed. Every projection onto the
first \(n\) coordinates is surjective: otherwise its image is a
proper subspace of \(\mathbb F_p^n\), annihilated by a nonzero
linear functional, giving a dependence relation among
\(\chi_1,\ldots,\chi_n\).
Full finite-coordinate projections make the image dense in
the product topology; closedness makes it all of \(K_p\).
\end{proof}
We now combine the preceding results to obtain the following fundamental
conclusion about \(\B(G)\).
\begin{proposition}\label{main:allB}
Every infinite compact Hausdorff abelian group has \(\B(G)\).
\end{proposition}

\begin{proof}
By Lemma~\ref{main:infinitedual}, the dual group \(\widehat G\) is
infinite. We divide the argument into three exhaustive cases.

Suppose first that \(\widehat G\) contains a character \(\chi\) of
infinite order. The image \(\chi(G)\) is a compact, hence closed,
subgroup of \(\mathbb T\). Since \(\chi\) has infinite order, this image
is infinite. Every proper closed subgroup of \(\mathbb T\) is finite:
the inverse image of a closed subgroup under
\(\mathbb R\to\mathbb R/\mathbb Z\) is a closed subgroup containing
\(\mathbb Z\), hence either \(\mathbb R\) or \((1/M)\mathbb Z\) for
an integer \(M\geq1\). So \(\chi(G)=\mathbb T\). Thus
\(\chi\colon G\to\mathbb T\) is a
continuous surjective homomorphism. Proposition~\ref{main:circle} gives
\(\B(\mathbb T)\), and Lemma~\ref{main:quotient} transfers this property
through \(\chi\), yielding \(\B(G)\).

It remains to consider the case in which every character has finite
order. If these orders are unbounded, Lemma~\ref{main:sampling} applies
directly and gives \(\B(G)\).

Finally, suppose that the orders of all characters are bounded by some
integer \(J\). Then
\[
 L=\operatorname{lcm}(1,\ldots,J)
\]
annihilates \(\widehat G\), so \(\widehat G\) has bounded exponent.
Since \(\widehat G\) is infinite, Lemma~\ref{main:bounded} gives a prime
\(p\) and a subgroup
\[
 \bigoplus_{j\geq1}\mathbb Z/p\mathbb Z
 \subseteq\widehat G.
\]
By Lemma~\ref{main:primequotient}, this subgroup determines a continuous
quotient map
\[
 q\colon G\longrightarrow
 K_p:=\prod_{j\geq1}\mathbb Z/p\mathbb Z.
\]
If \(p=2\), Proposition~\ref{main:walsh} gives \(\B(K_2)\); if \(p\) is
odd, Proposition~\ref{main:oddinput} gives \(\B(K_p)\). In either case,
Lemma~\ref{main:quotient} transfers the property from \(K_p\) to \(G\).
Thus \(G\) has property \(\B(G)\) in all three cases.
\end{proof}

\section{Construction of a continuous divergent function}\label{sec:construction}
The category argument in Section~\ref{sec:category} will show that
arbitrarily small divergent functions are generic. We first give an
explicit construction that, in addition, permits the Fourier
coefficients at any prescribed finite set of frequencies to vanish.

\begin{proposition}\label{main:gluing}
If \(G\) is infinite and has \(\B(G)\), then, for every finite
\(F\subseteq\widehat G\) and every \(\rho>0\), there is an
\(f\in\D(G)\) with \(\|f\|_\infty<\rho\) and
\(\widehat f=0\) on \(F\).
\end{proposition}

\begin{proof}
\emph{Choice of blocks.}
First we construct an example of norm less than one.
Lemma~\ref{main:infinitedual} ensures that the dual is infinite,
which will be required for spectral avoidance.
Given nonzero preceding blocks \(P_j\), choose
\[
 0<c_n<2^{-n},\qquad
 c_n<\tfrac12\min\{|\widehat P_j(\gamma)|:
                    j<n,\ \widehat P_j(\gamma)\ne0\},
\]
omitting the second condition when \(n=1\).
The minimum is positive, since it concerns finitely many nonzero
coefficients. By \(\B(G)\), choose \(Q_n\) with
\[
 \|Q_n\|_\infty\leq1,\qquad
 \mu_G\{T^*_GQ_n>(n+2)/c_n\}>1-2^{-n}.
\]
 Choose \(\eta_n\) so that its translated spectrum
avoids \(F\) and all earlier spectra, using
Lemma~\ref{main:modulation}, and put \(P_n=c_n\eta_nQ_n\).

\emph{Coefficient separation and convergence.}
The spectra are pairwise disjoint. Since \(|\widehat Q_n|\leq1\),
\begin{equation}\label{main:bands}
 \|P_n\|_\infty\leq c_n,\qquad
 0<|\widehat P_n(\gamma)|\leq c_n
       <\tfrac12|\widehat P_j(\lambda)|
\end{equation}
for every \(j<n\) and all nonzero coefficients in question.
Scaling and modulation give
\[
 \mu_G(E_n)<2^{-n},\qquad E_n=\{x:T^*_GP_n(x)\leq n+2\}.
\]

The series \(f=\sum_nP_n\) converges uniformly, defines a continuous
function, and has norm less than one. Termwise integration is
valid under uniform convergence. Thus every nonzero Fourier
coefficient of \(f\) is exactly the unique coefficient in its
block, and all coefficients on \(F\) vanish.

\emph{Pointwise divergence.}
For \(x\notin E_n\), choose a nonzero coefficient level
\(\delta_n(x)\) of \(P_n\) such that
\(|T_{\delta_n(x)}P_n(x)|>n+2\).
Such a level exists by the coefficient-level formula for \(T^*_GP_n\)
in Section~\ref{sec:transfer};
moreover \(0<\delta_n(x)\leq c_n\).
Let
\[
 m_n=\min\{|\widehat P_n(\gamma)|:
                   \widehat P_n(\gamma)\ne0\}.
\]
Since \(\delta_n(x)\) is a coefficient level of \(P_n\), we have
\(\delta_n(x)\geq m_n\). For \(j<n\), \eqref{main:bands} gives
\(|\widehat P_j|>2c_n\geq\delta_n(x)\) on the nonzero coefficients
of \(P_j\). For \(m>n\), the choice of \(c_m\) gives
\[
 |\widehat P_m(\gamma)|\leq c_m<\tfrac12m_n<\delta_n(x)
\]
for every nonzero coefficient of \(P_m\). Consequently
\[
 T_{\delta_n(x)}f(x)
   =\sum_{j<n}P_j(x)+T_{\delta_n(x)}P_n(x).
\]
Moreover, by \eqref{main:bands} and the choice \(c_j<2^{-j}\),
\[
 \left|\sum_{j<n}P_j(x)\right|
 \leq\sum_{j<n}\|P_j\|_\infty
 \leq\sum_{j<n}c_j
 <\sum_{j=1}^{\infty}2^{-j}=1.
\]
Therefore,
\[
 \left|T_{\delta_n(x)}f(x)\right|
 \geq \left|T_{\delta_n(x)}P_n(x)\right|
      -\left|\sum_{j<n}P_j(x)\right|
 >(n+2)-1=n+1.
\]
Since \(\sum_n\mu_G(E_n)<\infty\), the first Borel--Cantelli lemma
implies that almost every \(x\) lies outside \(E_n\) for all sufficiently
large \(n\). For each such \(x\), the corresponding thresholds tend to
zero because \(0<\delta_n(x)\leq c_n<2^{-n}\). The preceding estimate
therefore proves the required unboundedness.

Finally, choose \(0<a<\rho\). By \eqref{main:scaling},
\[
 T_\delta(af)=aT_{\delta/a}f,
\]
so \(af\in\D(G)\). Since \(\|f\|_\infty<1\), we have
\[
 \|af\|_\infty<a<\rho,
\]
and scaling preserves the vanishing of the Fourier coefficients on
\(F\). Replacing \(f\) by \(af\) completes the proof.
\end{proof}

\section{Generic divergence and consequences}\label{sec:category}
In this section we show that if \(G\) is infinite and \(\B(G)\) holds,
then \(\D(G)\) is a dense \(G_\delta\) in \(C(G)\). The proof uses a
Baire category argument.
\subsection{Stability of threshold sums}

A positive threshold \(\delta\) is called \emph{regular for \(f\)}
if it differs from every coefficient magnitude \(|\widehat f(\gamma)|\).

\Needspace{6\baselineskip}
For the category argument, we need large threshold sums to persist under
small uniform perturbations. This is ensured by choosing thresholds away
from coefficient magnitudes.

\begin{lemma}\label{main:regular}
If \(\delta\) is regular for \(f\in C(G)\), then all sufficiently
small uniform perturbations \(g\) preserve its regularity and the
selected set \(S=\{\gamma:|\widehat f(\gamma)|\geq\delta\}\).
On that neighborhood,
\[
 \|T_\delta g-T_\delta f\|_\infty\leq |S|\,\|g-f\|_\infty.
\]
\end{lemma}

\begin{proof}
The set \(L=\{\gamma:|\widehat f(\gamma)|\geq\delta/2\}\)
is finite, and
\[
 d=\min\left(\{\delta/2\}\cup
 \{\,||\widehat f(\gamma)|-\delta|:\gamma\in L\,\}\right)>0.
\]
Every coefficient magnitude is at distance at least \(d\)
from \(\delta\). The bound
\[
 \big||\widehat g(\gamma)|-|\widehat f(\gamma)|\big|
 \leq|\widehat{g-f}(\gamma)|\leq\|g-f\|_\infty
\]
shows that perturbations of norm less than \(d/2\) preserve
the selected set and regularity. Summing differences over this
fixed finite set gives the estimate.
\end{proof}

\Needspace{6\baselineskip}
Restricting to regular thresholds loses no threshold sums, as the following
observation shows.

\begin{lemma}\label{main:replacement}
For any \(f\in C(G)\) and \(\delta>0\), there is a regular
\(\delta'\in(0,\delta)\) with \(T_{\delta'}f=T_\delta f\).
\end{lemma}

\begin{proof}
The set
\[
 \left\{ |\widehat f(\gamma)|:
          \frac{\delta}{2}\leq |\widehat f(\gamma)|<\delta \right\}
\]
is finite. Define
\[
 b=\max\left(
     \left\{\frac{\delta}{2}\right\}
     \cup
     \left\{ |\widehat f(\gamma)|:
          \frac{\delta}{2}\leq |\widehat f(\gamma)|<\delta \right\}
     \right).
\]
Then \(b<\delta\). Thus every \(\delta'\in(b,\delta)\) is regular
for \(f\), and no coefficient magnitude lies in
\([\delta',\delta)\). Consequently,
\[
 T_{\delta'}f=T_\delta f.
\]
\end{proof}

\subsection{The category conclusion}

\Needspace{6\baselineskip}
We can now combine this stability with the finite examples to obtain
divergence on a dense \(G_\delta\) subset of \(C(G)\).

\begin{proposition}\label{main:category}
If \(G\) is infinite and has \(\B(G)\), then \(\D(G)\)
is a dense \(G_\delta\) in \(C(G)\).
\end{proposition}

\begin{proof}
For integers \(m,k\geq1\), let \(U_{m,k}\) consist of functions
admitting \(s\geq1\) regular thresholds
\(\delta_1,\ldots,\delta_s\in(0,1/k)\) with
\begin{equation}\label{main:U}
 \mu_G\left\{\max_{1\leq j\leq s}|T_{\delta_j}f|>m\right\}>1-2^{-k}.
\end{equation}
\emph{Openness.}
Suppose that \(f\in U_{m,k}\). Choose regular thresholds
\(\delta_1,\ldots,\delta_s\) for which \eqref{main:U} holds, and set
\[
 M_f=\max_{1\leq j\leq s}|T_{\delta_j}f|.
\]
Since
\[
 \{M_f>m\}=\bigcup_{\ell\geq1}\{M_f>m+1/\ell\},
\]
continuity from below yields some \(\alpha>0\) such that
\[
 \mu_G\{M_f>m+\alpha\}>1-2^{-k}.
\]
For each \(j\), let
\[
 S_j=\{\gamma\in\widehat G:
             |\widehat f(\gamma)|\geq\delta_j\}.
\]
By Lemma~\ref{main:regular}, if \(g\) is sufficiently close to \(f\),
then each \(\delta_j\) remains regular for \(g\), its selected set
remains \(S_j\), and
\[
 \|T_{\delta_j}g-T_{\delta_j}f\|_\infty
 \leq |S_j|\,\|g-f\|_\infty.
\]
Define
\[
 M_g=\max_{1\leq j\leq s}|T_{\delta_j}g|.
\]
Then
\[
 \|M_g-M_f\|_\infty
 \leq \max_{1\leq j\leq s}|S_j|\,\|g-f\|_\infty.
\]
Thus \(\|M_g-M_f\|_\infty<\alpha\) whenever \(g\) is sufficiently
close to \(f\), and hence
\[
 \{M_f>m+\alpha\}\subseteq\{M_g>m\}.
\]
Therefore \(g\in U_{m,k}\), proving that \(U_{m,k}\) is open.

\emph{Density.}
Fix \(f_0\in C(G)\) and \(r>0\).
Characters contain constants, separate points, and are closed
under products and conjugation. The Stone--Weierstrass theorem
\cite[Theorem~1.2.4(a)]{Rudin} supplies a character polynomial \(q\)
with \(\|q-f_0\|_\infty<r/2\).
Choose \(0<c<\min(r/2,1/k)\), also smaller than every nonzero
coefficient magnitude of \(q\), if \(q\ne0\).
By \(\B(G)\), choose \(P\) with norm at most one and
\[
 \mu_G\{T^*_GP>(m+\|q\|_\infty+1)/c\}>1-2^{-k}.
\]
By Lemmas~\ref{main:infinitedual} and \ref{main:modulation},
modulate to avoid the spectrum of \(q\), and put
\(h=c\eta P\), \(g=q+h\). Then \(\|g-f_0\|_\infty<r\).
Let \(\lambda_1,\ldots,\lambda_s\) be the distinct nonzero
coefficient magnitudes of \(h\). They are at most \(c<1/k\),
and spectral disjointness and the choice of \(c\) give
\[
 T_{\lambda_j}g=q+T_{\lambda_j}h.
\]
On the stated large set their maximum modulus exceeds \(m+1\).
For each \(j\), Lemma~\ref{main:replacement} supplies a threshold
\(\delta'_j\in(0,\lambda_j)\), regular for \(g\), with
\[
 T_{\delta'_j}g=T_{\lambda_j}g,\qquad
 0<\delta'_j<\lambda_j\leq c<1/k.
\]
Thus the sums are unchanged and all thresholds remain admissible.
They establish \eqref{main:U}, proving density.

The complete metric space \(C(G)\) satisfies Baire's theorem;
separability is not required. Hence
\[
 \mathcal R=\bigcap_{m\geq1}\bigcap_{k\geq1}U_{m,k}
\]
is a dense \(G_\delta\).
For \(f\in\mathcal R\) and fixed \(m\), the measurable sets
\[
 E_{m,k}=\{x:\sup_{0<\delta<1/k}|T_\delta f(x)|>m\}
\]
decrease with \(k\) and have measure greater than \(1-2^{-k}\).
Continuity from above gives
\[
 \mu_G\left(\bigcap_{k\geq1}E_{m,k}\right)
 =\lim_{k\to\infty}\mu_G(E_{m,k})
 \geq\lim_{k\to\infty}(1-2^{-k})=1.
\]
The decay of the exceptional bound is essential at this step.
Intersecting also over \(m\) proves \(\mathcal R\subseteq\D(G)\).

Conversely, if \(f\in\D(G)\), for fixed \(m,k\) the union of
the superlevel sets for thresholds in \((0,1/k)\) has measure one.
There are only countably many distinct sums, by
Lemma~\ref{main:thresholds}. Each can be represented at a regular
threshold still in this interval, by Lemma~\ref{main:replacement}.
Enumerate the resulting superlevel sets and take their increasing
finite unions. Their union has measure one, so continuity from below
gives a finite union of measure greater than \(1-2^{-k}\). The
corresponding finite regular thresholds show that \(f\in U_{m,k}\).
It follows that \(\D(G)=\mathcal R\).
\end{proof}

We now have all the ingredients needed to prove the main result.
\begin{proof}[Proof of Theorem~\ref{main:theorem}]
For infinite \(G\), combine Propositions~\ref{main:allB},
\ref{main:gluing}, and \ref{main:category}.
For finite \(G\), the dual is finite and a threshold below every
nonzero coefficient magnitude selects the whole expansion:
\(T_\delta f=f\) for all sufficiently small \(\delta>0\).
The same is immediate for \(f=0\).
\end{proof}

\subsection{Scope and consequences}
We conclude with a consequence of Theorem~\ref{main:theorem} and a
related structural observation.
\begin{corollary}\label{main:subgroup}
For every infinite subgroup \(\Lambda\subseteq\widehat G\), every
finite set \(F\subseteq\widehat G\), and every \(\rho>0\), there is an
\(f\in\D(G)\) such that
\[
 \|f\|_\infty<\rho,\qquad
 \supp\widehat f\subseteq\Lambda,
 \qquad \widehat f(\gamma)=0,\quad \gamma\in F.
\]
\end{corollary}

\begin{proof}
Let \(N=\Lambda^\perp\). Pontryagin duality identifies
\(\widehat{G/N}\) with \(\Lambda\): the double annihilator is
the closure of \(\Lambda\), which is already closed in the
discrete dual. Thus \(G/N\) is infinite. Apply the theorem on
this quotient, avoiding the given finite set intersected with
\(\Lambda\), and pull back by Lemma~\ref{main:quotient}.
\end{proof}

\begin{remark}
Every individual \(f\in C(G)\) factors through a compact metrizable
abelian quotient. Indeed, choose polynomials \(q_n\to f\) uniformly.
The subgroup \(\Lambda\) generated by their spectra is countable.
The evaluation map \(x\mapsto(\lambda(x))_{\lambda\in\Lambda}\)
has compact image in the metrizable group \(\mathbb T^\Lambda\).
Every \(q_n\), and hence \(f\), is constant on its fibers.
The descended function is continuous since a continuous surjection
from compact to Hausdorff is a quotient map.
This observation is compatible with, but not needed for, the
nonseparable Baire argument above.
\end{remark}

\begingroup
\sloppy

\endgroup

\begin{thebibliography}{99}
\bibitem{CF98}
A.~C\'ordoba and P.~Fern\'andez,
\emph{Convergence and divergence of decreasing rearranged Fourier series},
SIAM J. Math. Anal. \textbf{29} (1998), 1129--1139.
DOI: \nolinkurl{10.1137/S0036141097320705}.
\bibitem{E11}
S.~A. Episkoposian,
\emph{On greedy algorithms with respect to generalized Walsh system},
arXiv:1109.3806, 2011.
\bibitem{EG12}
S.~A. Episkoposian and M.~G. Grigorian,
\emph{\(L^p\)-convergence of greedy algorithm by generalized Walsh system},
J. Math. Anal. Appl. \textbf{389} (2012), 1374--1379.
DOI: \nolinkurl{10.1016/j.jmaa.2012.01.014}.
\bibitem{EGG24}
S.~A. Episkoposian, M.~G. Grigorian, and T.~M. Grigorian,
\emph{On the universal pair with respect to the generalized Walsh system},
Adv. Oper. Theory \textbf{9} (2024), article 3.
DOI: \nolinkurl{10.1007/s43036-023-00301-w}.
\bibitem{G73}
J.~Gosselin, \emph{Almost everywhere convergence of Vilenkin--Fourier series},
Trans. Amer. Math. Soc. \textbf{185} (1973), 345--370.
\bibitem{GY75}
J.~A. Gosselin and W.~S. Young,
\emph{On rearrangements of Vilenkin--Fourier series which preserve
almost everywhere convergence},
Trans. Amer. Math. Soc. \textbf{209} (1975), 157--174.
\bibitem{GS18}
M.~G. Grigoryan and S.~A. Sargsyan,
\emph{Almost everywhere convergence of greedy algorithm with respect to
Vilenkin system},
J. Contemp. Math. Anal. \textbf{53} (2018), 331--345.
DOI: \nolinkurl{10.3103/S1068362318060043}.
\bibitem{HR74}
E.~Hewitt and K.~A. Ross,
\emph{Rearrangements of \(L^r\) Fourier series on compact abelian groups},
Proc. London Math. Soc. (3) \textbf{29} (1974), 317--330.
DOI: \nolinkurl{10.1112/plms/s3-29.2.317}.
\bibitem{K96}
T.~W. K\"orner, \emph{Divergence of decreasing rearranged Fourier series},
Ann. of Math. (2) \textbf{144} (1996), 167--180.
\bibitem{K99}
T.~W. K\"orner, \emph{Decreasing rearranged Fourier series},
J. Fourier Anal. Appl. \textbf{5} (1999), 1--19.
DOI: \nolinkurl{10.1007/BF01274186}.
\bibitem{K06}
T.~W. K\"orner, \emph{Hard summation, Olevskii, Tao and Walsh},
Bull. Lond. Math. Soc. \textbf{38} (2006), 705--729.
DOI: \nolinkurl{10.1112/S002460930601890X}.
\bibitem{KO97}
S.~Kostyukovsky and A.~Olevskii,
\emph{Note on decreasing rearrangement of Fourier series},
J. Appl. Anal. \textbf{3} (1997), 137--142.
DOI: \nolinkurl{10.1515/JAA.1997.137}.
\bibitem{N07}
M.~Nielsen,
\emph{An example of an almost greedy uniformly bounded orthonormal basis
for \(L^p(0,1)\)},
J. Approx. Theory \textbf{149} (2007), 188--192.
DOI: \nolinkurl{10.1016/j.jat.2007.04.011}.
\bibitem{O75}
A.~M. Olevskii,
\emph{Fourier Series with Respect to General Orthogonal Systems},
Ergebnisse der Mathematik und ihrer Grenzgebiete, vol.~86,
Springer-Verlag, Berlin--New York, 1975.
\bibitem{O22}
G.~G. Oniani,
\emph{On the divergence sets of Fourier series in systems of characters
of compact abelian groups},
Math. Notes \textbf{112} (2022), 100--108.
DOI: \nolinkurl{10.1134/S0001434622070112}.
\bibitem{Rudin}
W.~Rudin, \emph{Fourier Analysis on Groups},
Interscience Publishers, New York, 1962.
\end{thebibliography}
\end{document}